\documentclass[11pt, oneside]{amsart}
\usepackage{amsfonts, amstext, amsmath, amsthm, amscd, amssymb}
\usepackage{color,graphics,graphicx,enumerate, comment,setspace,pinlabel,hyperref,caption}
\usepackage[all]{xy}
\usepackage[paper=a4paper, text={138mm,208mm},centering]{geometry}

\xyoption{import,dvips,rotate}

\DeclareMathOperator{\im}{im}

\DeclareMathOperator{\coker}{coker}

\DeclareMathOperator{\Hom}{Hom}

\DeclareMathOperator{\Ext}{Ext}
\DeclareMathOperator{\Tor}{Tor}

\DeclareMathOperator{\ann}{ann}

\DeclareMathOperator{\closure}{cl}

\newcommand{\eps}{\varepsilon}

\def\A{\mathbb{A}}

\def\C{\mathbb{C}}

\def\FF{\mathbb{F}}
\def\H{\mathbb{H}}
\def\R{\mathbb{R}}
\def\Z{\mathbb{Z}}

\def\Q{\mathbb{Q}}

\def\NN{{\mathfrak{N}}}

\def\p{{\mathfrak{p}}}

\def\lmat{\left(\begin{smallmatrix}}
\def\rmat{\end{smallmatrix}\right)}

\def\sm{\setminus}

\def\co{\colon\thinspace}

\theoremstyle{plain}
\newtheorem{theorem}{Theorem}[section]
\newtheorem{proposition}[theorem]{Proposition}
\newtheorem{lemma}[theorem]{Lemma}
\newtheorem{corollary}[theorem]{Corollary}

\newtheorem{question}[theorem]{Question}

\theoremstyle{definition}
\newtheorem{definition}[theorem]{Definition}
\newtheorem{example}[theorem]{Example}

\theoremstyle{remark}

\newtheorem*{remark}{Remark}

\begin{document}

\title[Double Witt groups]{Double Witt groups}
\author{Patrick Orson} 
\address{Department of Mathematics\\University of Durham\\United Kingdom}
\email{patrick.orson@durham.ac.uk}

\def\subjclassname{\textup{2010} Mathematics Subject Classification}
\expandafter\let\csname subjclassname@1991\endcsname=\subjclassname
\expandafter\let\csname subjclassname@2000\endcsname=\subjclassname


\begin{abstract}
The difference between slice and doubly-slice knots is reflected in algebra by the difference between metabolic and hyperbolic Blanchfield linking forms. We exploit this algebraic distinction to refine the classical Witt group of linking forms by defining a `double Witt group' of linking forms. We calculate the double Witt group for Dedekind domains and precisely determine its relationship to the classical Witt group. Finally, we prove that the double Witt group of Seifert forms is isomorphic to the double Witt group of Blanchfield forms.
\end{abstract}

\maketitle

\section{Introduction}

An $n$-knot $K:S^n\hookrightarrow S^{n+2}$ is \emph{slice} if it is the intersection of an $(n+1)$-knot $J$ and the equator sphere $S^{n+2}\subset S^{n+3}$. $K$ is called \emph{doubly-slice} if $J$ is moreover the unknot. For any odd dimensional knot $K$, the infinite cyclic cover of the knot complement has a middle-dimensional linking pairing called the \emph{Blanchfield form} of $K$. It is a non-singular linking form on a $P$-torsion $\Z[\Z]$-module, where here $P$ is the set of $p\in\Z[\Z]$ which augment to $\pm 1$. When $K$ is slice, the Blanchfield form admits a maximally self-annihilating submodule called a \emph{lagrangian} and when $K$ is doubly-slice the Blanchfield form decomposes as the direct sum of two lagrangians. In the first case, such a pairing is called \emph{metabolic} and in the second case \emph{hyperbolic} (see Levine \cite{MR1004605}). For $n>1$, the problem of detecting if a knot is slice has been solved: when $n$ is even,  Kervaire \cite{MR0189052} showed that all knots are slice, and when $n=2k+1$ is odd Kervaire \cite{MR0189052} and Levine \cite{MR0179803,MR0246314} showed that a knot is slice if and only if its Blanchfield form is metabolic. Thus the final stage in the solution to the slice problem was calculating the group $W^\eps(\Z[\Z],P)$, $\eps=(-1)^k$, of non-singular Blanchfield forms modulo metabolic Blanchfield forms. Levine \cite{MR0246314} and Stoltzfus \cite{MR0467764} calculated this to be \[W^\eps(\Z[\Z],P)\cong \bigoplus_\infty\Z\oplus \bigoplus_\infty(\Z/2\Z)\oplus \bigoplus_\infty(\Z/4\Z).\]

In contrast, and even when $n>1$, the doubly-slice problem is still quite poorly understood. Stoltzfus \cite{MR521738} has suggested the approach of working with non-singular Seifert forms modulo hyperbolic Seifert forms. In this paper we follow a similar line of thought by introducing more general tools for investigating the difference between metabolic and hyperbolic linking forms. These tools are called \emph{double Witt groups}.

For a commutative Noetherian ring $A$ and a multiplicative subset $S$ defining a ring localisation, we make the first general definition of the group of linking forms modulo hyperbolic linking forms. We call this  the \emph{double Witt group} of linking forms $DW^\eps(A,S)$. In order to prove that this is a group, we need to make the further assumption that there is an element $s\in A$ such that $s+\overline{s}=1$. We call such an element a \emph{half-unit}.

Building on work of Levine \cite{MR564114}, we prove in Section \ref{sec:DWgroups} a primary decomposition theorem which can be used to calculate some double Witt groups and also measures the difference between the Witt and double Witt groups.

\medskip

\noindent{\bf Theorem (\ref{thm:multisignature} and \ref{cor:forget})}\,\,
{\sl If $A$ is a Dedekind domain containing a half-unit then there is an isomorphism of abelian groups} \[DW(A,A\sm\{0\})\cong \bigoplus_{\p=\overline{\p}}\bigoplus_{l=1}^\infty W^{v_p}(A/\p),\quad v_p\co =\left\{\begin{array}{lll}\eps&&\text{$l$ even,}\\u_p\eps&&\text{$l$ odd,}\end{array}\right.\]{\sl where $\p$ denotes an involution invariant prime ideal with uniformiser $p=u_p\overline p$. Denoting the $(\p,l)$th component of the above decomposition by $\sigma_{\p,l}$, the (surjective) forgetful functor from the double Witt group to the Witt group is given by} \[DW(A,A\sm\{0\})\to W(A,A\sm\{0\});\qquad(T,\lambda)\mapsto \bigoplus_{\p=\overline{\p}}\left(\sum_{\text{$l$ odd}}\sigma_{\p,l}(T,\lambda)\right).\]

\medskip

In other words, for each primary component of the Witt group, the double Witt group has an infinite family of components. Half of these components (the $l$ even half) are invisible to the Witt group and the other half (the $l$ odd half) are combined into a sum and so cannot be distinguished from one-another in the Witt group. However, all components are visible to the double Witt group.

In Section \ref{sec:knots}, we then turn to the particular case of Blanchfield linking forms. Of course, in this case $A=\Z[\Z]$, so there is no half-unit in the ring. We show how to sidestep this issue and define a double Witt group in this setting as well. There is a well-known connection between Blanchfield forms and the more commonly used knot invariant called a \emph{Seifert form} for $K$. In particular the Witt groups of Seifert forms and of Blanchfield forms are isomorphic to one-another. We will prove the following using an operation called \emph{algebraic covering}:

\medskip

\noindent{\bf Theorem (\ref{thm:covering})}\,\,
{\sl If $R$ is any commutative Noetherian ring with unit, there is an isomorphism from the double Witt group of Seifert forms over $R$ to the double Witt group of Blanchfield forms over $(R[\Z],P)$}\[\widehat{DW_\eps}(R)\xrightarrow{\cong}DW^{-\eps}(R[\Z],P).\]

\medskip

In particular, combining Theorem \ref{thm:covering} with our primary decompositions and calculations in Section \ref{sec:DWgroups}, we may calculate the double Witt group of Seifert forms for any field coefficients, recovering a result of Stoltzfus \cite{MR521738}.

\section{Algebraic conventions and localisation}

We first lay down some general algebraic conventions with which we work. Notably, we will restrict ourselves to working over a \emph{commutative} ring $A$. As our main topological application will involve the group ring $\Z[\Z]$, this will be sufficient for our purposes.

In the following, $A$ will always be a commutative Noetherian ring with a unit, half-unit and involution. The involution is denoted\[\overline{\phantom{A}}\co A\to A;\qquad a\mapsto \overline{a}.\]Using the involution we define a way of switching between left and right modules, which permits an efficient way of describing sesquilinear pairings between left $A$--modules. A left $A$--module $P$ may be regarded as a right $A$--module $P^t$ by the action \[P^t\times A\to P^t;\qquad (x,a)\mapsto \overline{a}x.\]Similarly, a right $A$--module $P$ may be regarded as a left $A$--module $P^t$. Unless otherwise specified, the term `$A$--module' will refer to a left $A$--module. Given two $A$--modules $P,Q$, the tensor product is an abelian group denoted $P^t\otimes_A Q$. We will sometimes write simply $P\otimes Q$ to ease notation, but the right $A$--module structure $P^t$ is implicit, so that for example $x\otimes ay=\overline{a}x\otimes y$.

In the following, $S\subset A$ will always be a \emph{multiplicative} subset, that is a set with the following properties:

\begin{enumerate}[(i)]
\item $st\in S$ for all $s,t\in S$,
\item $sa=0\in A$ for some $s\in S$ and $a\in A$ only if $a=0\in A$,
\item $\overline{s}\in S$ for all $s\in S$,
\item $1\in S$.
\end{enumerate}The \emph{localisation of $A$ away from $S$} is $S^{-1}A$, the $A$--module of equivalence classes of pairs $(a,s)\in A\times S$ under the relation $(a,s)\sim (b,t)$ if and only if there exists $c\neq0\in A$ such that $c(at-bs)=0\in A$. We say the pair $(A,S)$ \emph{defines a localisation} and denote the equivalence class of $(a,s)$ by $a/s\in S^{-1}A$.


If $P$ is an $A$--module denote $S^{-1}P\co=S^{-1}A\otimes_AP$ and write the equivalence class of $(a/s)\otimes x$ as $ax/s$. Similarly, if $f\co P\to Q$ is a morphism of $A$--modules then there is induced a morphism of $S^{-1}A$--modules $S^{-1}f=1\otimes f\co S^{-1}P\to S^{-1}Q$. If $S^{-1}P\cong0$ then the $A$--module $P$ is \emph{$S$--torsion}. Using condition (ii) in the assumptions for $S$, the natural morphism $i\co A\to S^{-1}A$ sending $a\mapsto a/1$ induces a short exact sequence of $A$--modules \[0\to A\to S^{-1}A\to S^{-1}A/A\to 0.\]$i\co A\to S^{-1}A$ is moreover an injective ring morphism, but in general $S^{-1}A/A$ is not a ring. For an $A$--module $P$, the induced morphism $i\co P\to S^{-1}P$ is not generally injective. This is measured by the exact sequence\[0\to \Tor^A_1(S^{-1}A,P)\to \Tor_1^A(S^{-1}A/A,P)\to P\to S^{-1} P\to (S^{-1}A/A)\otimes_A P.\]From this, we see that $i\co P\to S^{-1}P$ is injective if and only if $\Tor^A_1(S^{-1}A/A,P)$ vanishes. This happens, for instance, when $P$ is a projective module. Define the \emph{$S$--torsion of $P$} to be\[TP:=\ker(P\to S^{-1}P).\]

\subsection*{Torsion modules and duality}

Define a category\[\A(A)=\{\text{finitely generated (\text{f.g.}), projective $A$--modules}\},\] with $A$--module morphisms. An $A$--module $Q$ has \emph{homological dimension $m$} if it admits a resolution of length $m$ by f.g.\ projective $A$--modules, i.e.\ there is an exact sequence\[0\to P_m\to P_{m-1}\to\dots\to P_0\to Q\to 0,\]with $P_i$ in $\A(A)$. If this condition is satisfied by all $A$--modules $Q$ we say $A$ is of homological dimension $m$. If $(A,S)$ defines a localisation, define a category \[\H(A,S)=\{ \text{f.g.\ $S$--torsion $A$--modules of homological dimension 1}\}\]with $A$--module morphisms. $\A(A)$ has a good notion of duality, coming from the $\Hom$ functor, and $\H(A,S)$ has a corresponding good notion of `torsion duality' as we now explain.

Given $A$--modules $P$, $Q$, we denote the additive abelian group of $A$--module homomorphisms $f\co P\to Q$ by $\Hom_A(P,Q)$. The \emph{dual} of an $A$--module $P$ is the $A$--module \[P^*:=\Hom_A(P,A)\]where the action of $A$ is $(a,f)\mapsto (x\mapsto f(x)\overline{a})$. If $P$ is in $\A(A)$, then there is a natural isomorphism\[\sm-\co P^t\otimes Q\xrightarrow{\cong}\Hom_A(P^*,Q);\qquad x\otimes y\mapsto (f\mapsto \overline{f(x)}y).\]In particular, using the natural $A$--module isomorphism $P\cong P^t\otimes A$, there is a natural isomorphism \[P\xrightarrow{\cong} P^{**};\qquad x\mapsto (f\mapsto \overline{f(x)}).\]Using this, for any $A$--module $Q$ in $\A(A)$ and $f\in\Hom_A(Q,P^*)$ there is a \emph{dual morphism}\[f^*\co P\to Q^*;\qquad x\mapsto (y\mapsto \overline{f(y)(x)}).\]

\begin{lemma}\label{lem:ext1}If $T$ is a f.g.\ $A$--module with homological dimension 1 and $T^*=0$ then there is a natural isomorphism of $A$--modules $T\cong \Ext^1_A(\Ext^1_A(T,A),A)$.
\end{lemma}

\begin{proof}Taking the dual of a projective resolution $P_1\xrightarrow{d} P_0\to T$ we obtain a projective resolution for $\Ext^1_A(T,A)$\[T^*=0\to P_0^*\xrightarrow{d^*} P_1^*\to \Ext^1_A(T,A)\to 0.\]Taking the dual again we obtain an exact sequence\[0\to (\Ext^1_A(T,A))^*\to P_1^{**}\xrightarrow{d^{**}} P_0^{**}\to \Ext^1_A(\Ext^1_A(T,A),A).\]The natural isomorphisms $P_i\cong P_i^{**}$ for $i=1, 2$ in particular imply that $(\Ext^1_A(T,A))^*=\ker{d^{**}}\cong\ker{d}=0$. A natural isomorphism of projective resolutions induces a natural isomorphism $T\cong \Ext^1_A(\Ext^1_A(T,A),A)$ (independent of the resolution chosen).
\end{proof}

\begin{lemma}\label{lem:ext2}If $(A,S)$ defines a localisation and $T$ is a f.g.\ $A$--module that is $S$--torsion, then $T^*=0$ and there is a natural isomorphism\[\Ext^1_A(T,A) \cong\Hom_A(T,S^{-1}A/A).\]
\end{lemma}

\begin{proof}As $T$ is f.g.\ there exists $s\in S$ such that $sT=0$, so clearly $T^*=0$. Similarly $\Hom_A(T,S^{-1}A)=0$. Using the short exact sequence of $A$--modules \[0\to A\to S^{-1}A\to S^{-1}A/A\to 0,\] we obtain the long exact sequence for $\Ext$\[0\to \Hom_A(T,S^{-1}A/A)\to \Ext_A^1(T,A)\to \Ext_A^1(T,S^{-1}A)\]and as localisation is exact we have $\Ext^1_A(T,S^{-1}A)\cong S^{-1}\Ext^1_A(T,A)=0$ as $T$ is $S$--torsion.
\end{proof}

Lemmas \ref{lem:ext1} and \ref{lem:ext2} show that, proceeding as in the category $\A(A)$ we may define a duality and the same results will follow. Precisely, the \emph{torsion dual} of a module $T$ in $\H(A,S)$ is the module\[T^\wedge:=\Hom_A(T,S^{-1}A/A)\]in $\H(A,S)$ with the action of $A$ given by $(a,f)\mapsto (x\mapsto f(x)\overline{a})$. There is a natural isomorphism\[T\xrightarrow{\cong} T^{\wedge\wedge};\qquad x\mapsto (f\mapsto \overline{f(x)}),\]and for $R,T$ in $\H(A,S)$, $f\in\Hom_A(R,T^\wedge)$ there is a \emph{torsion dual morphism}\[f^\wedge\co T\to R^\wedge;\qquad x\mapsto(y\mapsto\overline{f(y)(x)}).\]

\section{Double Witt groups}\label{sec:DWgroups}

Linking forms and the Witt groups of linking forms for various rings are important tools in surgery obstruction theory (see Wall \cite{MR0156890}, Ranicki \cite{MR620795}) and in the theory of knot concordance (see Levine \cite{MR0246314}, Kearton \cite{MR0385873}). Indeed in certain situations, the Witt group of linking forms is thought of as the \emph{algebraic} knot concordance group. Stoltzfus \cite{MR521738} applied this concept to the idea of \emph{double knot concordance} (see Sumners \cite{MR0290351}) to define an \emph{algebraic} double knot concordance group $CH^\eps(\Z)$ (which we call later $\widehat{DW}_\eps(\Z)$). In contrast to the algebraic knot concordance group, Stoltzfus' refined group is still not calculated.

In this section we expand on Stoltzfus' idea to make the first general definition of what we call the \emph{double Witt group of linking forms} for a localisation $(A,S)$. In the case that $A$ is a Dedekind domain we will calculate this group in terms of previously known invariants. In particular, we will show that each signature in the Witt group of linking forms over $(A,S)$ corresponds to a countably infinite family of signatures in the double Witt group. We determine the precise manner in which the double Witt group signatures collapse to form the component of the single Witt group.

\subsection{Forms and linking forms}

A standard references for the algebra of forms and classical Witt groups is Milnor-Husemoller \cite{MR0506372}. The language we use for the general setting of rings with involution is based on that found in Ranicki \cite[1.6, 3.4]{MR620795} although we caution that our terminology, particularly later on regarding lagrangian submodules, differs slightly. Also, our use of the word `split' in reference to forms and linking forms is entirely different to Ranicki's use.

\begin{definition}An \emph{$\eps$--symmetric form over $A$} is a pair $(P,\theta)$ consisting of an object $P$ of $\A(A)$ and an injective $A$--module morphism $\theta\co P\hookrightarrow P^*$ such that $\theta(x)(y)=\eps\overline{\theta(y)(x)}$ for all $x,y\in P$ (equivalently $\theta=\eps\theta^*$). A form $(P,\theta)$ is \emph{non-singular} if $\theta$ is an isomorphism. A form induces a sesquilinear pairing also called $\theta$\[\theta\co P\times P\to A;\qquad (x,y)\mapsto \theta(x,y):=\theta(x)(y).\]A morphism of $\eps$--symmetric forms $(P,\theta)\to (P',\theta')$ is an $A$--module morphism $f\co P\to P'$ such that $\theta(x)(y)=\theta'(f(x))(f(y))$ (equivalently $\theta=f^*\theta' f$), it is an isomorphism when $f$ is an $A$--module isomorphism. The set of isomorphism classes of $\eps$--symmetric forms over $A$, equipped with the addition $(P,\theta)+(P',\theta')=(P\oplus P',\theta\oplus \theta')$ forms a commutative monoid \[\NN^\eps(A)=\{\text{$\eps$--symmetric forms over $A$}\}.\]
\end{definition}

\begin{definition}\label{def:linking}Suppose $(A,S)$ defines a localisation. An \emph{$\eps$--symmetric linking form over $(A,S)$} is a pair $(T,\lambda)$ consisting of an object $T$ of $\H(A,S)$ and an injective $A$--module morphism $\lambda\co T\hookrightarrow T^\wedge$ such that $\lambda(x)(y)=\eps\overline{\lambda(y)(x)}$ for all $x,y\in T$ (equivalently $\lambda=\eps\lambda^\wedge$). A linking form $(T,\lambda)$ is \emph{non-singular} if $\lambda$ is an isomorphism. A linking form induces a sesquilinear pairing also called $\lambda$\[\lambda\co T\times T\to S^{-1}A/A;\qquad (x,y)\mapsto \lambda(x,y):=\lambda(x)(y).\]A morphism of $\eps$--symmetric linking forms $(T,\lambda)\to (T',\lambda')$ is an $A$--module morphism $f\co T\to T'$ such that $\lambda(x)(y)=\lambda'(f(x))(f(y))$ (equivalently $\lambda=f^\wedge\lambda' f$). $f$ is an isomorphism of forms when $f$ is an $A$--module isomorphism. The set of isomorphism classes of $\eps$--symmetric linking forms over $A$, equipped with the addition $(T,\lambda)+(T',\lambda')=(T\oplus T',\lambda\oplus \lambda')$ forms a commutative monoid \[\NN^\eps(A,S)=\{\text{$\eps$--symmetric linking forms over $(A,S)$}\}.\]
\end{definition}

\begin{definition}[Terminology]When $\eps=1$ we will omit $\eps$ from the terminology e.g.\ \emph{symmetric form} indicates $(+1)$-symmetric form.
\end{definition}

\subsection*{Linking form decomposition in a Dedekind domain}Recall that a \emph{Dedekind domain} is an integral domain for which every module has homological dimension 1. An equivalent definition is that a Dedekind domain is an integral domain in which every non-zero proper ideal factors uniquely as a product of prime ideals. Over a Dedekind domain $A$, a torsion $A$--module $T$ has annihilator $\ann(T ) = \p_1^{l_1} \p_2^{l_2} ...\p_m^{l_m}$ for some prime ideals $\p_i\subset A$, so that there is a natural isomorphism of $A$--modules\[T\cong\bigoplus_{j=1}^m T_{\p_j};\qquad T_{\p_j}:=A_{\p_j}\otimes_A T\cong \p_1^{l_1}... \p_{j-1}^{l_{j-1}}\p_{j+1}^{l_{j+1}} ...\p_m^{l_m}T.\]A \emph{principal ideal domain} (PID) is an integral domain in which every ideal is generated by a single element. A PID is always a Dedekind domain but the converse is not true in general. A \emph{local ring} is a ring with a unique maximal ideal.

Suppose for now that $A$ is an integral domain, that $\p\subset A$ is a prime ideal and that $T$ is in $\H(A,A\sm\{0\})$. $T$ is called \emph{$\p$-primary} if $\p^dT=0$ for some $d>0$. Suppose there is an element $p\in\p\sm\p^2$, which we call a \emph{uniformiser}. Define\[\p^\infty:=\{up^k\,|\,\text{$u$ is a unit, }k>0\}\subset A,\]which is a multiplicative subset if we assume $\p=\overline{\p}$. In this case, $T$ is $\p$-primary if and only if $T$ is $\p^\infty$-torsion. If $\p$ is principal, any generator $p$ of $\p$ works as a uniformiser. Moreover, if $A$ is a Dedekind domain then every prime ideal $\p$ admits a uniformiser.

How do linking forms interact with prime ideals? If $\p\subset A$ is a prime ideal for a Dedekind domain $A$ and $(T,\lambda)$ is a non-singular $\eps$--symmetric linking form over $(A,A\sm\{0\})$ then, by considering our algebraic convention for the $A$--module structure of a torsion dual, we see the restriction of $\lambda_\p:=\lambda|_{T_\p}$ defines an isomorphism\[\lambda_{\p}\co T_\p\xrightarrow{\cong} T_{\overline{\p}}^\wedge.\]So if $\p\neq \overline{\p}$ then there is a non-singular linking form\[\left(T_\p\oplus T_{\overline{\p}},\lambda|_{T_\p\oplus T_{\overline{\p}}}\right).\]And if $\p=\overline{\p}$ there is a non-singular linking form $(T_\p,\lambda_\p)$. If $\p=\overline{\p}$ is a prime ideal in a Dedekind domain then the ring $A_\p=(\p\sm\{0\})^{-1}A$, called the \emph{localisation away from the prime $\p$}, is a local ring and a PID, with maximal ideal generated by the class in $A_\p$ of any uniformiser $p\in A$ of $\p$. If $T$ is $\p$-torsion there is a natural isomorphism of $A$--modules $T\cong A_\p\otimes_A T$ inducing an equivalence of categories between $\H(A,\p^\infty)$ and $\H(A_\p,\p^\infty)$. We have just shown the following for Dedekind domains:

\begin{proposition}\label{prop:dedekinddecomp}If $A$ is a Dedekind domain and $(T,\lambda)$ is a non-singular $\eps$--symmetric linking form over $(A,A\sm\{0\})$, then the primary decomposition of $T$ induces a natural decomposition into non-singular linking forms\[(T,\lambda)\cong\left(\bigoplus_{\p=\overline{\p}} (T_\p,\lambda_\p)\right)\oplus\left(\bigoplus_{\p\neq\overline{\p}}\left(T_\p\oplus T_{\overline{\p}},\lambda|_{T_\p\oplus T_{\overline{\p}}}\right)\right)\]If $\p=\overline{\p}$ for all prime $\p\subset A$ there is induced a natural isomorphism of commutative monoids\[\NN^\eps(A,A\sm\{0\})\cong\bigoplus_{\p}\NN^\eps(A,\p^\infty)\cong \bigoplus_\p\NN^\eps(A_\p,\p^\infty).\]
\end{proposition}

We now focus on those $\p$-primary $A$--modules where $\p=\overline{\p}$. We will now assume that $A$ is a local ring and a PID (this is justified by making the shift from $(A,\p^\infty)$ to $(A_\p,\p^\infty)$ in general). Let $T$ be a f.g.\ torsion $A_\p$-module. There are naturally defined summands $T_l$ of $T$ such that there are isomorphisms of $A_\p$-modules\[T_l\cong\bigoplus A_\p/(p^lA_\p),\qquad T\cong\bigoplus_{l=1}^\infty T_l,\] $\ann(T_l)=\p^l$, $pT_l=\ker(p^{l-1}\co T_l\to T_l)$ and $T_l$ can be regarded as a free $A_\p/(p^{l}A_\p)$-module.

However, the natural decomposition of modules $T\cong\bigoplus_{l=1}^\infty T_l$ does not extend \emph{naturally} to a decomposition of a linking form $(T,\lambda)$. Here is a way to do it non-naturally:

\begin{proposition}\label{prop:nonnatural}Let $A$ be a local ring and a PID, with maximal ideal $\p=\overline{\p}$. Suppose $(T,\lambda)$ is an $\eps$--symmetric linking form over $(A,\p^\infty)$, then there exist $\eps$--symmetric linking forms $(T_l,\lambda_l)$ over $(A,\p^\infty)$ such that there is a (non-natural) isomorphism\[(T,\lambda)\cong\bigoplus_{l=1}^\infty(T_l,\lambda_l).\]
\end{proposition}

\begin{proof} Choose a generator $p\in A$ of $\p$ and let $d$ be the largest integer such that $p^{d-1} T\neq 0$. Write $j\co T_d\hookrightarrow T$ for the inclusion. We claim that the restriction $\lambda_d:=j^\wedge\lambda j$ of $\lambda$ to $T_d$ defines a non-singular linking form $(T_d,\lambda_{d})$ and that there is a (non-natural) choice of isomorphism $(T, \lambda) \cong(T_d, \lambda_d)\oplus(T',\lambda_{T'})$ where $T':=T/T_d$. Note that if $d'$ is the largest integer such that $p^{d'-1}T'=0$ then $d'<d$ so the Lemma is proved by iteration.

There is an $A$--module isomorphism $T_d^\wedge\cong T^*_d:=\Hom_{A/p^dA}(T_d,A/p^dA)$ so that both $T_d$ and $T_d^\wedge\cong T^*_d$ have the structure of free $(A/p^dA)$-modules. $\lambda_d\in\Hom_A(T_d,T_d^\wedge)\cong\Hom_{A/p^dA}(T_d,T_d^*)$ is an $A$--module isomorphism if and only if it is an $(A/p^dA)$-module isomorphism. An element $a \in A/p^dA$ is a unit if and only if $[a] \in A/pA$ is a unit, hence the $(A/p^dA)$-module morphism $\lambda_d$ is an isomorphism if and only if the $(A/pA)$-module morphism \[[\lambda_d]\co T_d/pT_d \to T_d^*/pT_d^*\] is an isomorphism. But as $d$ is the largest integer for which $p^{d-1}T \neq 0$, we have $T_d/pT_d = T/pT$, so in fact \[[\lambda_d] = [\lambda] \co  T/pT \to T^\wedge/pT^\wedge.\] As $\lambda\co T\to T^\wedge$ is an $A$--module isomorphism, so $[\lambda]$ is an $(A/pA)$-module isomorphism. Therefore, tracing the line of reasoning backwards, we see that $(T_d,\lambda_d)$ is non-singular.

As $j$ was originally a split injection we make a choice of splitting so that $T'=T/T_d$ and $T\cong T_d\oplus T'$. Restricting $\lambda$ to $T'$ via this isomorphism defines a non-singular linking form $(T',\lambda')$ as claimed.
\end{proof}

\begin{remark}Note that \emph{a choice of splitting of the form} was made when we wrote $(T, \lambda) \cong(T_d, \lambda_d)\oplus(T', \lambda')$. This choice is not natural. In general, when there is no half-unit in the ring, it is not even true that the monoid $\NN^\eps(A_\p,\p^\infty)$ decomposes according to the decompositions $T\cong \bigoplus_{l=1}^\infty T_l$. For instance within $\NN^\eps(\Z,\Z\sm\{0\})$, the localisation away from the prime $2\Z$ has forms which interrelate among the different values of $l$, see Wall \cite{MR0156890} and Kawauchi-Kojima \cite{MR594531}.
\end{remark}

What can be salvaged from this non-canonical construction? By restricting to Dedekind domains and assuming that there is a half-unit in the ring we will derive in Theorem \ref{thm:MDT2} a canonical decomposition for the monoid of linking forms. Our decomposition will be with respect to Levine's notion \cite{MR564114} of an \emph{auxiliary $l$-form}. This, in turn, is a generalisation of Wall's invariant $(\rho_k(G),b')$ (see \cite[\textsection 5]{MR0156890}) when $G$ is a finite abelian group.

\begin{definition}Let $A$ be an integral domain and suppose $(p)=\p$ is a principal ideal. For each $m\geq 0$, define an $A$--module\[K_m(T)=\ker(p^m\co T\to T),\]easily seen to be independent of the choice of $p\in A$. Then there is a filtration\[\{0\}=K_0\subseteq K_1\subseteq K_2\subseteq\dots\subseteq K_{d-1}\subseteq K_d=K_{d+1}=\dots\]which terminates as $A$ was assumed Noetherian. Each $K_m$ is $\p$-primary, and $T$ is $\p$-primary if and only if $K_d=T$. The subquotients $J_{l}:=K_{l}/K_{l-1}$ are $A/\p$-modules and are used to define the \emph{$l$th auxiliary module} over $A/\p$\begin{equation}\label{eq:auxiliary}\Delta_l(T):=\coker(p\co J_{l+1}\to J_{l})\cong\frac{K_l}{K_{l-1}+pK_{l+1}}\cong \frac{p^lK_{l+1}}{p^{l+1}K_{l+2}},\end{equation}which is natural in $T$.

Now suppose $\p=\overline{\p}$. Given an $\eps$--symmetric linking form $(T,\lambda)$ over $(A,\p^\infty)$, define for each $l>0$ the \emph{$l$th auxiliary form}, $(\Delta_l(T),b_l(\lambda))$ over $A/\p$ given by \[b_l([x],[y]):=p^{l-1}\lambda(x,y)\in A/\p\]where $x,y\in K_l$ are choices of representative of $[x],[y]\in J_l/pJ_{l+1}$ and $A/\p\subset S^{-1}A/A$ via $x\mapsto x/p$.
\end{definition}

Here is an interesting example of a linking form with a non-trivial symmetry that illustrates the definitions so far.

\begin{example}\label{lem:andrewripoff}Let $\FF$ be a field with involution and let $A=\FF[z,z^{-1}]$ with involution extended linearly from $\overline{z}=z^{-1}$. This is the Laurent polynomial ring with field coefficients and it is a principal ideal domain. Monic polynomials $p(z)\in\FF[z,z^{-1}]$ generate involution invariant prime ideals $\p$ if and only if they are of the type $p(z)=u_p\overline{p}(z)=(z-a)$ where $\overline{a}a=1$ and $u_p=-az$. For any such $a$ and corresponding $p$, there is a linking form $(A/p^lA,\lambda)$ over $(A,\p^\infty)$, with \[\lambda(x,y)= x\overline{y}/(z-a)^l\in \FF(z)/\FF[z,z^{-1}].\]This is easily checked to be $(u_p)^l$-symmetric. The corresponding auxiliary form $(\Delta_l(A/p^lA),b_l(\lambda))$ is the standard inner product on the field $A/(z-a)A$ given by the symmetric pairing $(x,y)\mapsto x\overline{y}\in A/pA$.
\end{example}

A vital tool for our calculations is following theorem of Levine, which we will call the \emph{Main Decomposition Theorem}.

\begin{theorem}[Main Decomposition Theorem {\cite[Theorem 20.1]{MR564114}}]\label{thm:levine} Suppose $A$ is a unique factorisation domain containing a half-unit, $\p=(p)=(\overline{p})=\overline{\p}$ is a prime ideal and $A/\p$ is a Dedekind domain. If $(T,\lambda)$, $(T',\lambda')$ are non-singular $\eps$--symmetric linking forms over $(A,\p^\infty)$, then there is an isomorphism of linking forms $(T,\lambda)\cong(T',\lambda')$ if and only if there is an isomorphism of $A$--modules $T\cong T'$ that induces an isomorphism of forms $(\Delta_l(T),b_l(\lambda))\cong(\Delta_l(T'),b_l(\lambda'))$ for each $l$.
\end{theorem}

Using this, we will prove the following monoid decomposition theorem.

\begin{theorem}\label{thm:MDT2}Let $A$ be a local ring and a PID with maximal ideal $(p)=\p$ where $p=u_p\overline{p}$ for some unit $u_p$. Define $\NN^\eps_{\p,l}\subset \NN^\eps(A,\p^\infty)$ to be the submonoid consisting of those linking forms $(T,\lambda)$ with $T=T_l$. Then for each $l>0$ there is an isomorphism of commutative monoids\[\NN^\eps_{\p,l}\cong \NN^{v_p}(A/\p)\qquad v_p:=\left\{\begin{array}{lll}\eps&&\text{$l$ even,}\\u_p\eps&&\text{$l$ odd.}\end{array}\right.\]Furthermore, there is a natural isomorphism of commutative monoids\[\NN^\eps(A,\p^\infty)\cong\bigoplus_{l=1}^\infty\NN^{v_p}(A/\p).\]\end{theorem}

\begin{proof}Suppose $T=T_l$, $T'=T'_l$ for some $l>0$. Then non-singular linking forms $(T,\lambda)$, $(T',\lambda')$ over $(A,\p^\infty)$ have naturally defined auxiliary forms $(\Delta_l,b_l)$, $(\Delta'_l,b'_l)$ (respectively) over $A/\p$. Choose a basis $e_1,\dots,e_n$ of $\Delta_l$ and $e_1',\dots,e_m'$ of $\Delta_l'$, so that a morphism $f\co \Delta_l\to\Delta_l'$ is a matrix $(f_{ij})$ with respect to these bases. We may lift these to bases $[e_i]$ and $[e_i']$ of $T$ and $T'$ respectively (regarded as free $A/p^lA$--modules). This allows us to define an $A/p^lA$--module morphism $[f]\co T\to T'$ by $[f]([e_i])=\sum_{j}f_{ij}[e_j]$. If $m=n$, then $f$ is an isomorphism if and only if the determinant of $f$ is a unit in $A/\p$. But if $u\in A$, then $u\in A/\p$ is a unit if and only if $u\in A/\p^k$ is a unit. Hence if there is an isomorphism $f\co (\Delta_l,b_l)\cong(\Delta_l',b_l')$ then it is induced by an isomorphism $T\cong T'$. This argument, taken together with the Main Decomposition Theorem, shows there is then a well-defined injective morphism of monoids \begin{equation}\label{eq:monoidmorph}\NN^\eps_{\p,l}\to \NN^{u_p^l\eps}(A/\p);\qquad(T,\lambda)\mapsto (\Delta_l(T),b_l(\lambda)).\end{equation}To show it is surjective, take any non-singular $\eps$--symmetric form $(\Delta,b)$ over $A/\p$ with a choice of basis $e_1,\dots, e_n$ of $\Delta$. Using this basis, we define a free $A/\p^l$-module $T$ and an $A/\p^l$-module morphism $\lambda\co T\to T^*=\Hom_{A/\p^l}(T,A/\p^l)$ by $\lambda(x,y):=p^{l-1}b(x,y)$. Then $(\Delta_l(T),b_l(\lambda))=(\Delta,b)$ as required. Therefore the morphism (\ref{eq:monoidmorph}) is an isomorphism of monoids.

The Main Decomposition Theorem gives a well-defined, natural morphism of commutative monoids\[\NN^\eps(A,\p^\infty)\to\bigoplus_{l=1}^\infty\NN^{u_p^l\eps}(A/\p);\qquad (T,\lambda)\mapsto \bigoplus_{l=1}^\infty (\Delta_l(T),b_l(\lambda)).\]By Proposition \ref{prop:nonnatural}, for any non-singular $\eps$--symmetric linking form $(T,\lambda)$ over $(A,\p^\infty)$ we can make a choice of decomposition $(T,\lambda)\cong\bigoplus_{l=1}^\infty(T_l,\lambda_l)$, with $(T_l,\lambda_l)$ in $\NN^\eps_{\p,l}$. By the Main Decomposition Theorem, this isomorphism induces an isomorphism $(\Delta_l(T),b_l(\lambda))\cong(\Delta_l(T_l),b_l(\lambda_l))$ for each $l>0$. We showed earlier how to make a choice of lift of elements in $\NN^{u_p\eps}(A/\p)$ to elements in $\NN^\eps_{\p,l}\subset \NN^\eps(A,\p^\infty)$ and our choice of decomposition $(T,\lambda)\cong\oplus_{l=1}^\infty(T_l,\lambda_l)$ shows that $\NN^\eps(A,\p^\infty)$ is generated by such lifted elements.

We only need to show that there are no interrelations between the lifted generators for different values of $l$. But suppose we have a finite sum \[\sum_{l=1}^\infty\sum_k F_{l,k}=\sum_{l=1}^\infty\sum_k F'_{l,k}\in \NN^\eps(A,\p^\infty),\]where $F_{l,k},F'_{l,k}\in\NN^\eps_{\p,l}$ are non-singular linking forms for each $k$. Then writing the auxiliary forms of $F_{l,k}$ and $F'_{l,k}$ as $\Delta_{l,k}$ and $\Delta'_{l,k}$ respectively we must have for each $l>0$\[\sum_k\Delta_{l,k}=\sum_k\Delta'_{l,k}\in \NN^{u_p^l\eps}(A/\p).\]So any relation in $\NN^\eps(A,\p^\infty)$ results in a finite set of relations, indexed by $l$, each relation within a respective copy of $\NN^{u_p^l\eps}(A/\p)$. Conversely, any such finite set of relations on $\NN^\eps(A/\p)$, indexed by $l$, will recover a relation in $\NN^\eps(A,\p^\infty)$ by the Main Decomposition Theorem. Hence there is a relation in $\NN^\eps(A,\p^\infty)$ if and only if it is generated by a finite set of relations within respective $\NN^{u_p\eps}(A/\p)$, independent for each $l>0$. 

We finally have to modify the symmetries to the $v_p$ in the statement of the theorem, but we do this by a standard trick. If $u=\overline{u}^{-1}$ is a unit in a ring $R$ and $(K,\alpha)$ is a $(u^2\eps)$--symmetric form over $R$ then the pairing $\alpha_u(x,y):=\alpha(\overline{u}x,y)=\overline{u}\alpha(x,y)=u\eps\overline{\alpha(y,x)}=\eps\overline{\alpha(\overline{u}y,x)}=\eps\overline{\alpha_u(y,x)}$ is $\eps$--symmetric. The assignment $(K,\alpha)\mapsto (K,\alpha_u)$ defines an isomorphism $\NN^{u^2\eps}(R)\cong \NN^\eps(R)$. Hence we may reduce our symmetries to the $v_p$ claimed. This completes the proof.\end{proof}

Our proof says nothing about what any of the generators of the monoids \emph{are}. We have not computed the monoids involved, hence the conspicuous lack of something like Hensel's Lemma as employed by Wall \cite[\textsection 5]{MR0156890}.

\subsection{Double Witt groups}\label{sec:witt}

We will now define several ways in which an $\eps$--symmetric form or linking form can be considered trivial, that all involve the idea a lagrangian submodule.

\begin{definition}A \emph{lagrangian} for a non-singular $\eps$--symmetric form $(P,\theta)$ over $A$ is a submodule $j \co  L \hookrightarrow P$ in $\A(A)$ such that the sequence\[0\to L\xrightarrow{j}P\xrightarrow{j^*\theta} L^*\to0\]is exact. As modules in the category $\A(A)$ are projective, all surjective morphisms split, and a lagrangian is always a direct summand. If $(P,\theta)$ admits a lagrangian it is called \emph{metabolic}. If $(P,\theta)$ admits two lagrangians $j_\pm\co L_\pm\hookrightarrow P$ (labelled ``$+$'' and ``$-$'') such that they are complementary as submodules\[\left(\begin{matrix}j_+\\j_-\end{matrix}\right)\co L_+\oplus L_-\xrightarrow{\cong} P,\]then the form is called \emph{hyperbolic}.
\end{definition}

\begin{remark}If $j\co L\hookrightarrow P$ is a submodule, then $L$ is a lagrangian if and only if \[j(L)=L^\perp:=\ker(j^*\theta\co T\to L^*)\subset T.\] In this case the form $\theta$ vanishes on $j(L)$ and $L$ is referred to as a `maximally self-annihilating submodule'. The reader might prefer to define a lagrangian in these terms.\end{remark}

\begin{lemma}\label{lem:splitishyp1}If $j\co L\hookrightarrow P$ is a lagrangian for a non-singular $\eps$--symmetric form $(P,\theta)$ over $A$ then for any splitting $k\co L^*\to P$ of $j^*\theta$, there is an isomorphism of forms\[(j\,\,k)\co \left(L\oplus L^*,\left(\begin{matrix}0&1\\\eps &k^*\theta k\end{matrix}\right)\right)\xrightarrow{\cong}(P,\theta).\] If $\theta=\psi+\eps\psi^*$ (for instance if $A$ contains a half-unit, set $\psi=s\theta$) then we may choose a splitting $k'$ such that there is an isomorphism of forms\[(j\,\, k')\co \left(L\oplus L^*,\left(\begin{matrix}0&1\\\eps&0\end{matrix}\right)\right)\xrightarrow{\cong}(P,\theta),\] in particular all metabolic $\eps$--symmetric forms $(P,\theta)$ over $A$ are hyperbolic in this case.
\end{lemma}

\begin{proof}See Ranicki \cite[2.2]{MR560997}.
\end{proof}

\begin{definition}Suppose $(A,S)$ defines a localisation. A \emph{(split) lagrangian} for a non-singular $\eps$--symmetric linking form $(T,\lambda)$ over $(A,S)$ is a submodule $j \co  L \hookrightarrow T$ in $\H(A,S)$ such that the sequence\[0\to L\xrightarrow{j}T\xrightarrow{j^\wedge\lambda} L^\wedge\to0\]is (split) exact. If $(T,\lambda)$ admits a (split) lagrangian it is called \emph{(split) metabolic}. If $(T,\lambda)$ admits two lagrangians $j_\pm\co L_\pm\hookrightarrow T$ such that they are complementary as submodules\[\left(\begin{matrix}j_+\\j_-\end{matrix}\right)\co L_+\oplus L_-\xrightarrow{\cong} T,\]then the form is called \emph{hyperbolic}.
\end{definition}

\begin{lemma}\label{lem:splitishyp2}If $A$ contains a half-unit then all split metabolic $\eps$--symmetric linking forms $(T,\lambda)$ over $(A,S)$ are hyperbolic.
\end{lemma}

\begin{proof}The proof of Lemma \ref{lem:splitishyp1} carries through exactly the same, using torsion duals.
\end{proof}

\begin{remark}In general, not assuming the presence of a half-unit, for a \textbf{linking form} there is a hierarchy:\[\text{hyperbolic}\,\,\subsetneq\,\,\text{split metabolic}\,\,\subsetneq\,\,\text{metabolic}.\]And for a \textbf{form}\[\text{hyperbolic}\,\,\subsetneq\,\,\text{split metabolic}\,\,=\,\,\text{metabolic}.\]Moreover, Lemmas \ref{lem:splitishyp1} and \ref{lem:splitishyp2} show that the presence of a half-unit destroys the distinction between split metabolic and hyperbolic for both forms and linking forms. More generally, if the form or linking form admits a quadratic extension, this distinction is destroyed.
\end{remark}

\begin{example}

Some examples to illustrate the differences between the types. \begin{enumerate}[(i)]
\item A symmetric form over $\Z$ that is metabolic but not hyperbolic: \[(P,\theta)=\left(\Z\oplus\Z,\left(\begin{array}{cc}0&1\\1&1\end{array}\right)\right)\] with lagrangian $\Z\oplus 0$. The property of taking only even values is preserved under isomorphisms of forms over $\Z$ (cf. `Type I' inner product spaces in Milnor-Husemoller \cite[4.2]{MR0506372}), so as $\theta((x,y),(x,y))=2xy+y^2$, we can deduce this is not isomorphic to the only possible hyperbolic form here, which is $\left(\Z\oplus \Z,\lmat0&a\\b&0\rmat\right)$ for $a,b\in\{\pm1\}$. 
\item A symmetric linking form over the local ring $(\Z_{\p},\p^\infty)$ with $\p=3\Z$ that is metabolic but not split metabolic: $(T,\lambda)=(\Z/9\Z,(a,b)\mapsto ab/9)$ with lagrangian $\Z/3\Z$. This is the only possible lagrangian and it is not a direct summand.
\item To see a symmetric linking form over $(\Z,\Z\sm\{0\})$ that is split metabolic but not hyperbolic we refer the reader to Banagl-Ranicki \cite{MR2189218}.
\end{enumerate}
\end{example}

\subsection{Witt groups: single or double?}

We will now construct the classical Witt group and our double Witt groups using the following basic construction.

\begin{definition}[Monoid construction]\label{def:monoid} Let $(M, +)$ be an commutative monoid and let $N$ be a submonoid of $M$. Consider the equivalence relation: for $m_1, m_2 \in M$, define $m_1 \sim m_2$ if there exists $n_1, n_2 \in N$ such that $m_1 + n_1 = m_2 + n_2$. Then the set of equivalence classes $M/\sim$ inherits a structure of abelian monoid via $[m] + [m'] := [m + m']$. It is denoted by $M/N$. Assume that for any element $m \in M$ there is an element $m' \in M$ such that $m + m' \in N$, then $M/N$ is an abelian group with $-[m] = [m']$. \emph{It is then canonically isomorphic to the quotient of the Grothendieck group of $M$ by the subgroup generated by $N$.}
\end{definition}

We wish to emphasise that even if the monoid construction $M/N$ is a group, the group does not necessarily have the property that $[m]=0\in M/N$ implies $m\in N$. We will often prove this as an additional property, but it does not come for free. The monoid construction is in general the group of \emph{stable isomorphism classes} in $M$, where one is allowed to `stabilise an isomorphism' by elements of the submonoid.

\begin{lemma}If $(P,\theta)$ is an $\eps$--symmetric form over $A$ then the form \[(P,\theta)+(P,-\theta)=(P\oplus P,\theta\oplus-\theta)\]is split metabolic.
\end{lemma}
\begin{proof}The diagonal $\lmat 1\\1\rmat\co P\to P\oplus P$ is a lagrangian with splitting $(1\,\,0) \co P\oplus P\to P$.
\end{proof}

\begin{lemma}\label{lem:halfunithyp}If $s\in A$ is a half-unit and $(P,\theta)$ is an $\eps$--symmetric form over $A$ then \[(P,\theta)+(P,-\theta)=(P\oplus P,\theta\oplus-\theta)\]is hyperbolic.
\end{lemma}

\begin{proof}$(P\oplus P, \theta\oplus -\theta)$ has lagrangians \[\begin{array}{lccrcl}(&1&1&)&\co &(P,0)\to (P\oplus P,\theta\oplus -\theta),\\(&\bar{s}&-s&)&\co &(P,0)\to (P\oplus P,\theta\oplus -\theta).\end{array}\]They are complementary as \[\left(\begin{matrix}s & 1\\\bar{s} & -1\end{matrix}\right)\left(\begin{matrix}1&1\\\bar{s}& -s\end{matrix}\right)=\left(\begin{matrix}1&1\\\bar{s}& -s\end{matrix}\right)\left(\begin{matrix}s & 1\\\bar{s} & -1\end{matrix}\right)=I.\]\end{proof}

The preceding lemmas hold completely analogously for linking forms. This justifies the following definitions:

\begin{definition}Suppose $(A,S)$ defines a localisation. The monoid constructions\[\begin{array}{rcl}
W^\eps(A)&=&\NN^\eps(A)/\{\text{metabolic forms}\}\\
W^\eps(A,S)&=&\NN^\eps(A,S)/\{\text{metabolic linking forms}\}
\end{array}\]are abelian groups called the \emph{$\eps$--symmetric Witt group} of $A$ and of $(A,S)$ respectively. The monoid construction\[\begin{array}{rcl}
DW^\eps(A,S)&=&\NN^\eps(A,S)/\{\text{hyperbolic linking forms}\}
\end{array}\] is an abelian group called the \emph{$\eps$--symmetric double Witt group} of $(A,S)$.
\end{definition}

\begin{remark}\label{clm:DWvsW}Of course we could have additionally defined a double Witt group of forms over $A$ in a similar fashion. But, by Lemma \ref{lem:splitishyp1}, the forgetful functor\[\NN^\eps(A)/\{\text{hyperbolic forms}\}\xrightarrow{\cong} W^\eps(A)\]is an isomorphism.
\end{remark}

\begin{remark}The first appearance of a double Witt group of a localisation $(A,S)$ is in Stoltzfus \cite[3.3]{MR521738}. Stoltzfus defines the `hyperbolic algebraic concordance groups of $\eps$--symmetric isometric structures' over a Dedekind domain $A$, which he denotes $CH^\eps(A)$ when the underlying modules are projective and denotes $CH^\eps(K/A)$ (for $K$ the fraction field of $A$) when the underlying modules are torsion. These groups are defined using a split metabolic definition but Stoltzfus shows \cite[3.2]{MR521738} that the resulting groups are double Witt groups in our sense. Indeed, for the projective case, Stoltzfus' group $CH^\eps(A)$ is isomorphic to the double Witt group $DW^\eps(A[z,z^{-1},(1-z)^{-1}],P)$ where $P=\{p(z)\in A[z,z^{-1}]\,|\,\text{$p(1)\in A$ is a unit}\}$.
\end{remark}

\subsection*{Calculating double Witt groups of linking forms}

We will assume for the rest of this section that $A$ is a Dedekind domain.

We now wish to calculate some double Witt groups of linking forms and analyse their relationship to Witt groups of linking forms via the forgetful map. Using Corollary \ref{cor:forget} and some standard Witt group calculations from the literature, we are able to give an answer to the following question:

\begin{question}\label{q:DWkernel}What is the kernel of the surjective forgetful morphism\[DW^\eps(A,S)\twoheadrightarrow W^\eps(A,S)\text{?}\]\end{question}

An obvious consequence of Proposition \ref{prop:dedekinddecomp} is that when $A$ is a Dedekind domain and $\p\neq\overline{\p}$ are prime ideals, then if $(T,\lambda)$ is a non-singular linking form over $(A,A\sm\{0\})$, the non-singular linking form $(T_\p\oplus T_{\overline{\p}},\lambda|_{T_\p\oplus T_{\overline{\p}}})$ is hyperbolic. The decomposition in Proposition \ref{prop:dedekinddecomp} then becomes a decomposition of Witt groups.

\begin{proposition}\label{prop:dedekinddecomp2}For $A$ a Dedekind domain, there are natural isomorphisms of abelian groups\begin{eqnarray*}W^\eps(A,A\sm\{0\})&\cong&\bigoplus_{\p=\overline{\p}}W^\eps(A,\p^\infty),\\
DW^\eps(A,A\sm\{0\})&\cong&\bigoplus_{\p=\overline{\p}}DW^\eps(A,\p^\infty).\end{eqnarray*}
\end{proposition}

In light of this, we will now specialise to $\p$-primary Witt groups for $\p=\overline{\p}$. First we mention a standard technique (cf.\ Ranicki \cite[4.2.1]{MR620795}) for manipulating Witt groups of linking forms over $(A,\p^\infty)$ called \emph{devissage}.

\begin{proposition}[Devissage]\label{prop:devissage}Let $A$ be a local ring with involution invariant maximal ideal $\p$. Suppose $(T,\lambda)$ is a non-singular $\eps$--symmetric linking form over $(A,\p^\infty)$ and that $T=T_l$ for some $l>0$, then
\[(T,\lambda)\sim\left\{\begin{array}{lcl}0&&\text{$l$ even}\\ (T',\lambda')&&\text{$l$ odd}\end{array}\right. \quad\in W^\eps(A,\p^\infty),\]where $(T',\lambda')$ is such that $p T'=0$.
\end{proposition}

\begin{proof}Choose a generator $p=u_p\overline{p}$ of $\p$ for some unit $u_p$. For $2k\geq l$, define a submodule $L=p^{k} T\subset T$. It is clear that $K_m\subseteq (p^mT)^\perp$. We claim moreover that $L^\perp = K_k(T)$. To see this, suppose $x\in (p^mT)^\perp$. Then $0=\lambda(x,p^my)=\lambda(\overline{p}^mx,y)$ for all $y\in T$. But $\lambda$ is injective, so indeed $x\in\ker(\overline{p}^m\co T\to T)=K_m$. 

Now we have\[L\subseteq L^\perp\subseteq p^{l-k}T\] and it is standard for linking forms (or easily checked) that $\lambda$ restricts to a non-singular $\eps$--symmetric linking form on the quotient $(L^\perp/L,\lambda)$ when $L\subseteq L^\perp$. There is a lagrangian\[L^\perp\hookrightarrow (T\oplus (L^\perp/L),\lambda\oplus -\lambda);\qquad x\mapsto (x,[x]),\]so that $(T,\lambda)\sim (L/L^\perp,\lambda)\in W^\eps(A,\p^\infty)$. In particular, when $l=2k$ we have $L^\perp/L$=0, and when $l=2k-1$ we have $p(L^\perp/L)\subset p(p^{k-1}T/L)=0$ so the result follows.
\end{proof}

\begin{proposition}\label{prop:residuefield}Let $A$ be a local ring with involution invariant maximal ideal $\p=(p)$ such that $u_p\overline{p}=p$ for some unit $u_p=\overline{u_p}^{-1}$. Then there is an isomorphism of abelian groups\[W^{u_p\eps}(A/\p)\cong W^\eps(A,\p^\infty).\]
\end{proposition}

\begin{proof}There is an equivalence of categories between the category of finitely generated $A$--modules $T$ such that $pT=0$ and the category of finite dimensional vector spaces over $A/\p$. Under which, for any f.g.\ $\p$-torsion $A$--module $T$, we observe that there is an isomorphism of $A$--modules \[\Hom_{A/\p}(T,A/\p)\xrightarrow{\cong}\Hom_A(T,\p^{-1}A/A);\qquad f\mapsto (x\mapsto f(x)/p).\]But consider that this interacts with symmetry in the following way. If $(T,\theta)$ is a $(u_p\eps)$--symmetric form over $A/\p$, then the corresponding $(T,\lambda)$ has \[\begin{array}{rcl}\eps\lambda^\wedge(x)(y)&=&\eps\overline{\lambda(y,x)}\\&=&\eps\overline{\theta(y,x)}/\overline{p}\\&=&u_p\eps\theta^*(x,y)/p=\lambda(x)(y).\end{array}\]This induces an isomorphism of commutative monoids $\NN^{u_p\eps}(A/\p)\cong \NN^\eps(A,\p)$. As the submonoids of metabolic forms are preserved under the isomorphism $\NN^{u_p\eps}(A/\p)\cong \NN^\eps(A,\p)$, there is an isomorphism of Witt groups after the respective monoid constructions.
\end{proof}

Combining Propositions \ref{prop:devissage} and \ref{prop:residuefield}, one obtains the following theorem first proved in the case that $1/2\in A$ by Karoubi \cite{MR0384894} and first proved generally by Ranicki \cite[4.2.1]{MR620795} (even without the assumption of a half-unit).

\begin{theorem}\label{thm:karoubi}Let $A$ be a Dedekind domain. Then there is an isomorphism of abelian groups\[\sigma^W\co W^\eps(A,A\sm\{0\})\xrightarrow{\cong} \bigoplus_{\p=\overline{\p}}W^{u_p\eps}(A/\p).\]The image $\sigma^W(T,\lambda)$ of a non-singular $\eps$--symmetric linking form is called the \emph{Witt multisignature}.
\end{theorem}

We will now prove that there is \emph{never} devissage in the double Witt groups over Dedekind domains. In other words, the double Witt relations will respect the decomposition of Theorem \ref{thm:MDT2}.

\begin{definition}Let $A$ is a Dedekind domain and $(T,\lambda)$ be a non-singular $\eps$--symmetric linking form over $(A,A\sm\{0\})$. For each involution invariant prime ideal $\p\subset A$, choose a uniformiser $p\in A$ such that $p=u_p\overline{p}\in A_\p$. Then for each $l>0$ we define the $(\p,l)$-signature of $(T,\lambda)$ to be the Witt class \[\sigma_{\p,l}(T,\lambda):=[(\Delta_l(T_\p),b_l(\lambda_\p))]\in W^{v_p}(A/\p),\qquad v_p:=\left\{\begin{array}{lll}\eps&&\text{$l$ even,}\\u_p\eps&&\text{$l$ odd.}\end{array}\right.\]
\end{definition}

\begin{theorem}\label{thm:multisignature}Let $A$ be a Dedekind domain conatining a half-unit $s$ and for each involution invariant prime ideal $\p\subset A$, choose a uniformiser $p\in A$ such that $p=u_p\overline{p}\in A_\p$. Then there are isomorphisms of abelian groups\[\begin{array}{rcl}DW^\eps(A,A\sm\{0\})&\xrightarrow{\cong}& \bigoplus_{\p=\overline{\p}} DW^\eps(A,\p^\infty)\\
&\xrightarrow{\cong}& \bigoplus_{\p=\overline{\p}}\bigoplus_{l=1}^\infty W^{v_p}(A/\p),\end{array}  \qquad v_p:=\left\{\begin{array}{lll}\eps&&\text{$l$ even,}\\u_p\eps&&\text{$l$ odd.}\end{array}\right.\]The composite isomorphism is given by the collection of $(\p,l)$-signatures $\sigma_{\p,l}(T,\lambda)$ and is called the \emph{double Witt multisignature} \[\sigma^{DW}\co DW^\eps(A,A\sm\{0\})\xrightarrow{\cong}\bigoplus_{\p=\overline{\p}}\bigoplus_{l=1}^\infty W^{v_p}(A/\p).\]
\end{theorem}

\begin{proof}The proof is given for the most part by our decompositions so far, the Main Decomposition Theorem and Theorem \ref{thm:MDT2}. The final isomorphism in the statement of the theorem uses the fact that $DW^{v_p}(A/\p)\cong W^{v_p}(A/\p)$ as $A/\p$ is a field.

The only thing we must still check is that a non-singular linking form $(T,\lambda)$ over $(A,\p^\infty)$ is hyperbolic if and only if $(\Delta_l(T),b_l(\lambda))$ is hyperbolic for each $l>0$. The proof of this is essentially given in Levine \cite[1.5]{MR1004605}.

Clearly if $T\cong L_+\oplus L_-$ for lagrangians $L_\pm$ then $\Delta_l(T)\cong\Delta_l(L_+)\oplus\Delta_l(L_-)$ and $\Delta_l(L_\pm)$ are lagrangians for $(\Delta_l(T),b_l(\lambda))$ for all $l>0$.

For the converse, suppose $(T,\lambda)$ is a non-singular linking form over $(A,\p^\infty)$ and choose a decomposition $(T,\lambda)\cong\bigoplus_{l>0}(T_l,\lambda_l)$ as in Proposition \ref{prop:nonnatural}. We will first check that if $(\Delta_l(T),b_l(\lambda))$ is split metabolic for some $l>0$ then the non-singular form $(T_l,\lambda_l)\in \NN^\eps_{\p,l}$ is split metabolic. By Proposition \ref{prop:nonnatural}, this will be enough to conclude that $(T,\lambda)$ is split metabolic.

As in the proof of Proposition \ref{prop:nonnatural} we prefer to think of $(T_l,\lambda_l)$ not as a linking form, but equivalently as a form over the ring $A/(p)^l$. As in the proof of Theorem \ref{thm:MDT2}, $(\Delta_l(T),b_l(\lambda))\cong(\Delta_l(T_l),b_l(\lambda_l))$ and we may lift a split lagrangian of $(\Delta_l(T),b_l(\lambda))$ to a split injection $j\co L\hookrightarrow T_l$ (not yet a lagrangian necessarily!). So there is an isomorphism of $A/(p)^l$-modules $L\oplus T_l/L\cong T_l$, under which identification we may write the linking form $\lambda_l=\lambda|_{T_l}$ as \[\left(\begin{matrix}f&g\\\eps g^*&h\end{matrix}\right)\co L\oplus (T_l/L)\to L^*\oplus (T_l/L)^*\cong (L\oplus (T_l/L))^*\](with $-^*=\Hom_{A/(p)^l}(-,A/(p)^l)$), for some morphisms $h=\eps h^*$ and $f=\eps f^*$, where $f$ is divisible by $p$, and $g$ must be an isomorphism (by non-singularity of $(T_l,\lambda_l)$). We assume without loss of generality that $T_l/L=L^*$ and that $g=1$ as we may modify the submodule $i\co T_l/L\hookrightarrow T_l$ to $ig^{-1}\co L^*\hookrightarrow T_l$. 

We now improve our split injection $j$ to the inclusion of a split lagrangian for $(T_l,\lambda_l)$. To do this we will use the fact that $f$ is divisible by $p$ as the base case of an induction on $l>k\geq 1$. Assume that $f=p^k\Phi$ for some $A$--module morphism $\Phi\co L\to L^*$ and take $s\in A$ such that $s+\overline{s}=1\in A/(p)^l$. We use this to modify the inclusion of the submodule $j\co L\hookrightarrow T_l$ to the split injection \[\left(j\,\,\, -i(\eps sp^k\Phi)^*\right)\co L\oplus L\hookrightarrow T_l.\]This new lagrangian submodule is still complementary to $i\co L^*\hookrightarrow T_l$ and the modification does not affect $h$ in our matrix. However, $f$ is modified to
\begin{eqnarray*}
&&f-(p^ks\Phi+\eps(p^ks\Phi)^*)+(p^{k}s\Phi)h(p^ks\Phi)^*\\
&=&f-(sf+\eps \overline{s}f^*)+p^{k}\overline{p}^ks\overline{s}\Phi h\Phi^*\\
&=&u_p^kp^{2k}s\overline{s}\Phi h\Phi^*,
\end{eqnarray*}
where we have used $\overline{p}^k=u_pp^k$. So the modification of $f$ is divisible by $p^{2k}$. By induction we may assume $f$ is divisible by $p^l$ and $j\co L\hookrightarrow T_l$ is a split lagrangian. Hence $(T_l,\lambda_l)$ is split metabolic as we claimed.

To complete the proof we must now assume $(\Delta_l(T),b_l(\lambda))$ admits a pair of complementary split lagrangians for each $l>0$ and conclude that $(T,\lambda)$ is hyperbolic. But for a fixed $l>0$ we may use the procedure above on each split lagrangian of $(\Delta_l(T),b_l(\lambda))$ to define a split lagrangian for $(T_l,\lambda_l)$. Moreover note that the induction stage of the above construction can be performed on each lagrangian \emph{independent of the other} (this was the observation that the inductive modification did not affect the $h$ morphism in the matrix). So we conclude that $(T_l,\lambda_l)$ is hyperbolic for each $l>0$ and hence $(T,\lambda)$ is hyperbolic.
\end{proof}

\begin{corollary}\label{cor:forget}If $A$ is a Dedekind domain containing a half-unit then the forgetful functor is given by \[DW^\eps(A,A\sm\{0\})\to W^\eps(A,A\sm\{0\});\qquad(T,\lambda)\mapsto \bigoplus_{\p=\overline{\p}}\bigoplus_{l \text{ odd}}\sigma_{\p,l}(T,\lambda).\]
\end{corollary}

\subsection*{Examples}

We now offer some examples of rings and localisations to which we can apply Theorem \ref{thm:multisignature} and Corollary \ref{cor:forget}.

\begin{example}The only double Witt group calculation prior to this work was performed by Kawauchi-Kojima \cite{MR594531}. By defining a complete system of invariants (which we do not detail here!) they are able to compute the monoid $\NN^+(\Z,\Z\sm\{0\})$ and hence the corresponding double Witt group (which they refer to as the `Witt group of linkings' \cite[Proposition 5.2]{MR594531}). Our results do not apply at the prime 2, but at odd primes we can recover their decomposition using Theorem \ref{thm:multisignature} to compute that \begin{gather*}{DW_p^l(\Z)\cong\left\{\begin{array}{lll}
\Z/2\Z\oplus\Z/2\Z&\,&p\equiv 1\,\,\text{(mod 4)},\\
\Z/4\Z&\,&p\equiv 3\,\,\text{(mod 4)},\end{array}\right.}\\
DW(\Z,(p)^\infty)\cong\bigoplus_{l=1}^\infty DW_p^l(\Z),\quad\text{for $p$ an odd prime.}\end{gather*}Now set $(A,S)=(\Z[\frac{1}{2}],A\sm\{0\})$. By Corollary \ref{cor:forget} we have $(T,\lambda)\in\ker(DW(A,S)\to W(A,S))$ if and only if for each prime $p\equiv 1\,\,\text{mod $4$}$ we have \[\sum_{l\,\,\text{odd}}\sigma_{(p),l}=(0,0)\in\Z/2\Z\oplus \Z/2\Z,\]and for each prime $p\equiv 3\,\,\text{mod $4$}$ we have \[\sum_{l\,\,\text{odd}}\sigma_{(p),l}=0\in\Z/4\Z.\]
\end{example}

We now move onto the example of Laurent polynomial rings over fields. All the terminology and results below concerning (single) Witt groups are well-known and borrowed from Ranicki \cite[39C]{MR1713074}. Only the statements concerning double Witt groups in our examples are original.

Suppose $\FF=\R$ or $\C$ where $\R$ has the trivial involution and $\C$ has the involution given by complex conjugation. The \emph{Laurent polynomial ring} $A=\FF[z,z^{-1}]$ has the involution extended linearly from $\FF$ by $\overline{z}=z^{-1}$. Define the set of involution invariant units in a commutative Noetherian ring $R$ to be \[U(R)=\{a\in R\,|\,\overline{a}a=1\}.\]Write the set of \emph{irreducible monic polynomials over $\FF$} as $\mathcal{M}(\FF)$ and define a subset as\[\{p(z)\in \mathcal{M}(\FF)\,|\,\overline{(p)}=(p)\}=: \overline{\mathcal{M}}(\FF)\subseteq\mathcal{M}(\FF).\]In other words, an irreducible monic polynomial $p(z)$ is in $\overline{\mathcal{M}}(\FF)$ if and only if there exists $u_p\in U(\FF[z,z^{-1}])$ such that $u_p\overline{p}=p$. It is well-known that the prime ideals of $\FF[z,z^{-1}]$ are in 1:1 correspondence with the elements of $\mathcal{M}(\FF)$. Define multiplicative subsets \begin{gather}P:=\{p(z)\in\FF[z,z^{-1}]\,|\,p(1)\neq 0\},\\ \tag*{\text{and}} Q:=\FF[z,z^{-1}]\sm \{0\}=P\cup (z-1)^\infty.\end{gather}Then $Q^{-1}\FF[z,z,^{-1}]=\FF(z)$, the fraction field.

Now consider Example \ref{lem:andrewripoff} where $p=(z-a)$ for $a\overline{a}=1$ and note that if we had taken an arbitrary $\eps$--symmetric linking form $(T,\lambda)$ over $(A,p^\infty)$, the $l$th auxiliary form $(\Delta_l(T),b_l(\lambda))$ over the field $A/pA$ is $v_p$-symmetric for \[v_p:=\left\{\begin{array}{lll}\eps&&\text{$l$ even,}\\u_p\eps&&\text{$l$ odd.}\end{array}\right.\]The isomorphism \[f\co A/pA\xrightarrow{\cong}\FF;\qquad z\mapsto a,\]affects $v_p$ in the following way\[f\co \left\{\begin{array}{lll}\eps&&\text{}\\u_p\eps&&\text{}\end{array}\right.\mapsto \qquad\left\{\begin{array}{lll}f(\eps)&&\text{$l$ even,}\\-a^2f(\eps)&&\text{$l$ odd.}\end{array}\right.\]But now by our standard trick from the end of the proof of Theorem \ref{thm:MDT2} we can modify the $-a^2f(\eps)$--symmetry to $-f(\eps)$--symmetry. As single Witt groups are blind to auxiliary forms when $l$ is even, one consequence is the isomorphism:\[W^\eps(\FF[z,z^{-1}],(a-z)^\infty)\cong W^{-f(\eps)}(\FF).\]

\begin{example}\label{ex:C}Note that for any algebraically closed $\FF$, we have a 1:1 correspondence of sets \[U(\FF)\to\overline{\mathcal{M}}(\FF);\qquad a\mapsto (z-a).\]

If $\FF=\C$ then recall (Ranicki \cite[39.22]{MR1713074}) that for $\eps=\pm1$ there is an isomorphism\[W^\eps(\C)\xrightarrow{\cong}\Z ,\]given by the signature of the hermitian pairing in the case $\eps=1$ and in the case $\eps=-1$ we send the skew-hermitian pairing $\theta(x,y)$ to the hermitian pairing $\theta(x,iy)$ and take the signature of this. By the discussion above and Example \ref{lem:andrewripoff} we have \[\begin{array}{rcl}W^\eps(\C[z,z^{-1}],Q)&\xrightarrow{\cong}&\bigoplus_{a\in S^1}\Z,\\
&&\\
DW^\eps(\C[z,z^{-1}],Q)&\xrightarrow{\cong}&\bigoplus_{a\in S^1}\bigoplus_{l=1}^\infty\Z.\end{array}\]and $(T,\lambda)$ is in the kernel of the forgetful map if and only if for each $a\in S^1$ we have\[\sum_{l\,\,\text{odd}}\sigma_{a,l}(T,\lambda)=0\in\Z.\]
\end{example}

\begin{example}\label{ex:R}If $\FF=\R$ then (by Ranicki \cite[39.23]{MR1713074}) we have\[\overline{\mathcal{M}}(\R)=\{(z-1)\}\cup\{(z+1)\}\cup\{p_\theta(z)\,|\,0<\theta<\pi\},\]where \[p_\theta(z)=(z- e^{i\theta})(z-e^{-i\theta}).\]The corresponding residue class fields are given by the rings with involution\[\begin{array}{rclr}\R[z,z^{-1}]/(z\pm 1)&\xrightarrow{\cong}&\R;&z\mapsto \mp1,\\
\R[z,z^{-1}]/(p_\theta(z))&\xrightarrow{\cong}&\C;&z\mapsto e^{i\theta}.\end{array}\]Hence we have\[\begin{array}{rcl}W^\eps(\R[z,z^{-1}],Q)&\cong&\left\{\begin{array}{lll}\bigoplus_{0<\theta<\pi}\Z&&\eps=1,\\
\Z\oplus\Z\oplus\bigoplus_{0<\theta<\pi}\Z&&\eps=-1.\end{array}\right.\\
&&\\
DW^\eps(\R[z,z^{-1}],Q)&\cong&\left\{\begin{array}{lll}\bigoplus_{0<\theta<\pi}\bigoplus_{l=1}^\infty\Z&&\eps=1,\\
(\bigoplus_{l=1}^\infty\Z)\oplus(\bigoplus_{l=1}^\infty\Z)\oplus(\bigoplus_{0<\theta<\pi}\bigoplus_{l=1}^\infty\Z)&&\eps=-1.\end{array}\right.\end{array}\]Again, $(T,\lambda)$ is in the kernel of the forgetful map if and only if \[\begin{array}{rcll}\eps=1\co &&&\sum_{l\,\,\text{odd}}\sigma_{\theta,l}(T,\lambda)=0\in\Z \qquad\text{for each $0<\theta<\pi$},\\
&&&\\
\eps=-1\co &&&\sum_{l\,\,\text{odd}}\sigma_{\pm1,l}(T,\lambda)=0\in\Z,\\
&&\text{and}& \sum_{l\,\,\text{odd}}\sigma_{\theta,l}(T,\lambda)=0\in\Z\qquad\text{for each $0<\theta<\pi$}.\end{array}
\]
\end{example}

\begin{example}\label{ex:Q}The Witt group for $\FF=\Q$ is also well-known. We refer the reader to Ranicki \cite[39.24]{MR1713074} for details of this group. In particular (and using the terminology of \cite[39.24]{MR1713074}), for each $(z\pm 1)\neq p(z)\in \overline{\mathcal{M}}(\Q)$ there is defined a natural number $t_{p(z)}$, and for each of $\eps=\pm1$ there are integer-valued signature invariants $\sigma^i_{p}$ for $i=1,\dots, t_{p(z)}$. Hence a non-singular $(\pm1)$-symmetric linking form $(T,\lambda)$ over $(\Q[z,z^{-1}],Q)$ determines an element of the kernel of the forgetful map only if for each $(z\pm 1)\neq p(z)\in \overline{\mathcal{M}}(\Q)$ and each $i=1,\dots,t_{p(z)}$\[\sum_{l\,\,\text{odd}}\sigma^i_{p,l}=0.\]
\end{example}

If we restrict to the multiplicative subset $P$ in Example \ref{ex:R}, we recover precisely the obstructions obtained by Levine \cite[1.7]{MR1004605} to a knot being \emph{doubly-slice}. Levine's obstructions arise from the $\R$-coefficient Blanchfield form of the knot (see Section \ref{sec:knots} for definition). Restricting to the multiplicative subset $P$ in Examples \ref{ex:C} and \ref{ex:Q} we obtain doubly-slice obstructions similarly. As far as we can tell, our obstructions over the coefficient rings $\C$ and $\Q$ are new, although we have not included them as a theorem in Section \ref{sec:knots} as they are very similar to those of Levine. According to Levine \cite{MR0246314}, and applying Theorem \ref{thm:covering}, we see that each of these obstructions is in fact realised by an odd-dimensional knot in every odd dimension.

\section{Seifert and Blanchfield forms}\label{sec:knots}

\begin{definition}An \emph{$\eps$--symmetric Seifert form} $(K,\psi)$ over $R$ is a f.g.\ projective $R$--module $K$ and a morphism of $R$--modules $\psi\co K\to K^*$ such that $\psi+\eps\psi^*$ is an isomorphism (note that this makes $(K,\psi+\eps\psi^*)$ a non-singular $\eps$--symmetric form). For a Seifert form $(K,\psi)$ we define an endomorphism $e=(\psi+\eps\psi^*)^{-1}\psi$  and we note that this gives $1-e=\eps(\psi+\eps\psi^*)^{-1}\psi^*$). A morphism of Seifert forms $g\co (K,\psi)\to(K',\psi')$ is a morphism of $R$--modules $g\co K\to K'$ such that $g^*\psi'g=\psi$, it is an isomorphism if $g$ is an $R$--module isomorphism.
\end{definition}

The action of $e$ on the underlying $R$--module $K$ of a Seifert form $(K,\psi)$ makes $K$ an $R[s]$--module where $s$ is a formal variable with action $s(x):=e(x)$ for $x\in K$. $R[s]$ is a ring with involution where we extend the involution from $R$ by $\overline{s}=1-s$. When we wish to remember the morphism by which $s$ acts, we will write $(K,e)$. The \emph{Seifert dual} of an $R[s]$--module $(K,e)$ is the $R[s]$--module $(K,e)^*=(K^*:=\Hom_R(K,R),1-e^*)$. An $R[s]$--submodule of $K$ is an $R$--submodule $j\co L\hookrightarrow K$ such that $ej(L)\subset j(L)$, then $L$ inherits an $R[s]$--module structure in the obvious way.

\begin{definition}\label{def:DWseifert}A \emph{(split) lagrangian} for a $\eps$--symmetric Seifert form $(K,\psi)$ over $R$ is a $R[s]$--submodule $j \co  L \hookrightarrow K$ such that the sequence in the category of $R[s]$--modules and Seifert duals\[0\to (L,e)\xrightarrow{j}(K,e)\xrightarrow{j^*(\psi_\eps+\psi^*)} (L,e)^*\to0\]is (split) exact, and $j^*\psi j=0$. If $(K,\psi)$ admits a (split) lagrangian it is called \emph{(split) metabolic}. If $(K,\psi)$ admits two lagrangians $j_\pm\co L_\pm\hookrightarrow K$ such that they are complementary as $R[s]$--submodules\[\left(\begin{matrix}j_+\\j_-\end{matrix}\right)\co L_+\oplus L_-\xrightarrow{\cong} K,\]then the form is called \emph{hyperbolic}. Although $R$ itself is not assumed to have a half-unit, the automorphism $s$ behaves like a half-unit and we may hence define the \emph{$\eps$--symmetric Witt and double Witt group of Seifert forms over $R$} respectively by\[\widehat{W}_\eps(R)\qquad\text{and}\qquad \widehat{DW}_\eps(R).\]
\end{definition}

The types of linking forms that arise in classical knot theory, called Blanchfield forms, are non-singular $\eps$--symmetric linking forms over $(R[z,z^{-1}],P)$ where $P$ is the set of \emph{Alexander polynomials} \[P:=\left\{p(z)\in R[z,z^{-1}]\,|\,p(1)\in R\text{ is a unit}\right\}.\]

The ring $R[z,z^{-1}]$ does not contain a half-unit necessarily. However, we will now make a well-known observation that will allow us to define double Witt groups $DW^\eps(R[z,z^{-1}], P)$ in the normal way. It is possible to formally adjoin the half-unit $(1-z)^{-1}$ to $R[z,z^{-1}]$ without affecting the category $\H(R[z,z^{-1}],P)$ or the corresponding categories of linking forms. Indeed, according to Ranicki \cite[10.21(iv)]{MR1713074}, there is an equivalence of exact categories\begin{equation}\label{eq:excision}\H(R[z,z^{-1}],P)\cong \H(R[z,z^{-1},(1-z)^{-1}],P).\end{equation}Under this equivalence, an object $T$ of $\H(R[z,z^{-1}],P)$ corresponds to a homological dimension 1, f.g.\ $R[z,z^{-1}]$--module $T$ such that $1-z\co T\to T$ is an isomorphism. 

\begin{proposition}\label{prop:cartmorph}The category of non-singular $\eps$--symmetric linking forms is unchanged under the equivalence of Equation \ref{eq:excision}. Under the equivalence of Equation \ref{eq:excision}, (split) lagrangians correspond to (split) lagrangians.

For each $n>1$, the categories of $P$--acyclic $n$--dimensional $\eps$--symmetric (Poincar\'{e}) complexes and of $P$--acyclic $(n+1)$--dimensional $\eps$--symmetric (Poincar\'{e}) pairs are unchanged under the equivalence of Equation \ref{eq:excision}.
\end{proposition}

\begin{proof}The equivalence of Equation \ref{eq:excision} comes from a special case of a general Cartesian morphism of localisations of rings with involution (see Ranicki \cite[p.\ 201]{MR620795}). The proof of Proposition \ref{prop:cartmorph} for general Cartesian morphisms can be found in Ranicki \cite[3.1.3, 3.6.2, 3.2.1]{MR620795}. See also Ranicki \cite[\textsection 4]{MR2058802}.
\end{proof}

Proposition \ref{prop:cartmorph} justifies the definition:

\[DW^\eps(R[z,z^{-1}],P):=DW^\eps(R[z,z^{-1},(1-z)^{-1}],P).\]

\subsection{Seifert and Blanchfield forms of a $(2k+1)$--knot}

An \emph{$n$--knot}, or \emph{knot} unless $n$ is to be specified, is an ambient isotopy class of oriented, locally flat topological embeddings $K\co S^n\hookrightarrow S^{n+2}$ (where all spheres are considered to have a preferred orientation already). In a standard abuse of notation we will also use the word knot to mean a particular $K$ in an ambient isotopy class and the image of $K$ in $S^{n+2}$. Any embedding $K\co S^{n}\hookrightarrow S^{n+2}$ has trivial normal bundle (2--plane bundles over spheres are trivial for $n>1$, and when $n=1$ consider that the normal bundle is oriented) and hence, by choosing a framing, we may excise a small, trivial tubular neighbourhood of the knot from $S^{n+2}$. Thus, the \emph{knot exterior} is the manifold with boundary \[(X_K,\partial X_K):=(\closure (S^{n+2}\sm (K(S^n)\times D^2)),S^n\times S^1)\]which has a preferred orientation coming from the ambient $S^{n+2}$. The knot exterior $X_K$ is homotopy equivalent to the \emph{knot complement} $S^{n+2}\sm K$ and hence has the homology of a circle $H_*(X_K)= H_*(S^1)$ by Alexander duality. The abelianisation $\pi_1(X_K)\to H_1(X_K)\cong\Z$ thus determines an \emph{infinite cyclic cover $\overline{X_K}$} with group of deck transformations $\Z$ generated by $z$, making $H_*(\overline{X_K})$ a finitely generated $\Z[z,z^{-1}]$-module. $H_*(\overline{X_K})$ is moreover torsion with respect to the set of Alexander polynomials $P$ (see Levine \cite[Corollary 1.3]{MR0461518}).

If there is a locally flat embedding of a manifold with boundary $(F^{n+1},S^n)\hookrightarrow S^{n+2}$ then we say the embedded $F$ is a \emph{Seifert surface} for the boundary knot. It was shown by many authors independently that every knot $K$ admits a Seifert surface $F^{n+1}$ (e.g.\ Kervaire \cite{MR0189052}, Zeeman \cite{MR0160218}). For $n\neq 2$, the unknot is characterised as the only knot which admits $D^{n+1}$ as a Seifert surface.

A knot has two very tractable homological invariants, called the Blanchfield and Seifert forms of the knot. For the sake of clarity and motivation for the subsequent algebra, we will write down the well-known relationships between the Blanchfield and Seifert forms. We will then use this description of the Seifert and Blanchfield forms to prove that the double Witt group of Seifert forms is isomorphic to the double Witt group of Blanchfield forms.

We turn first to the Blanchfield form of an $n$--knot $K$, the original version of which was defined by Blanchfield \cite{MR0085512}. In fact our description is based on that in Levine \cite{MR0461518}. Set $A=\Z[z,z^{-1}]$ and suppose $T$ is an $A$--module and that $\Hom_A(T,A)=0$. Set $t(T)$ to be the \emph{$\Z$--torsion} \[t(T)=\ker(T\to \Q[z,z^{-1}]\otimes_A T)\qquad \text{and} \qquad f(T)=T/t(T).\]Note that $f(T)$ may still have torsion with respect to the multiplicative subset $\Z[z,z^{-1}]\sm\{0\}$. Levine \cite{MR0461518} shows that for any module $T$ in $\H(A,P)$ the $\Z$--torsion and $\Z$--torsion free components are picked out as follows\[\begin{array}{rcl}\Ext_A^2(T,A))&\cong& t(T),\\
\Ext_A^1(T,A))&\cong& f(T).\\
\end{array}\]Now suppose $C$ is a bounded $P$--acyclic chain complex over $A$ and that there is a chain homotopy equivalence $\phi:C^{n-*}\simeq C$. Then by the universal coefficient spectral sequence collapse detailed in Levine \cite{MR0461518} we obtain an isomorphism\[\begin{array}{rcl}f(H^r(C;A))&\xrightarrow{\cong}& H^r(C;A)/\Ext^2_A(H_{r-2}(C;A),A)\\&\xrightarrow{\cong}&\Ext_A^1(H_{r-1}(C;A,A))\\ &\xrightarrow{\cong}& \Ext_A^1(H^{(n+1)-r}(C;A,A))\cong \Hom_A(f(H^{(n+1)-r}(C;A)),P^{-1}A/A).\end{array}\]When $C=C_*(\overline{X_K})$ is the singular chain complex of the universal cover of the knot exterior and $[X_K]\cap-=\phi:C^{n-*}\xrightarrow{\simeq} C$ is Poincar\'{e} duality, the above isomorphism is adjoint to the following pairing:

\begin{definition}\label{def:levblanch}The \emph{Blanchfield form} for a $(2k+1)$-knot $K$ is the non-singular $(-1)^{k}$--symmetric linking form over $(\Z[z,z^{-1}],P)$ defined by the pairing \[Bl\co f(H^{k+2}(\overline{X_K}))\times f(H^{k+2}(\overline{X_K}))\to P^{-1}A/A;\qquad (x,y)\mapsto p^{-1}\overline{\tilde{y}(\phi(x))},\]where $x,y\in C^{k+2}$, $\tilde{y}\in C^{k+1}$ and $p\in P$ such that $d^*\tilde{y}=py$.
\end{definition}

\begin{remark}Our definition is actually Poincar\'{e} dual to the usual definition of a Blanchfield form as defined by Levine \cite{MR0461518}.
\end{remark}

We look now at the Seifert form for an $n$--knot $K$. Suppose we have made a choice of Seifert surface $j\co F\hookrightarrow S^{n+2}$ for the knot $K$. The cohomology linking form on $H^{k+2}(S^{n+2})$ is clearly trivial but the underlying pairing given by linking is non-trivial on the chain level \[l\co C^{k+2}(S^{n+2})\times C^{k+2}(S^{n+2})\to \Z;\qquad (x,y)\mapsto x(a),\]where $da=y\cap[S^{n+2}]\in C_{k+1}$.

Choose a chain homotopy inverse $D\co C_{*}(S^{n+2})\to C^{n+2-*}(S^{n+2})$ to the chain-level cap-product $-\cap[S^{n+2}]$. If $P$ is a f.g.\ $\Z$--module, denote the torsion-free component by:\[f(P):=P/TP.\]

\begin{definition}\label{def:seifert}The \emph{Seifert form of $(F,K)$} is the (well-defined) $(-1)^{k+1}$--symmetric Seifert form $(f(H^{k+1}(F)),\psi)$ over $\Z$ given by\[\psi\co f(H^{k+1}(F))\times f(H^{k+1}(F))\to \Z;\qquad([u],[v])\mapsto l(x,y),\]where $u,v\in C^{k+1}(F)$ and $x=D(j_*(u\cap[F]))$, $\partial y=D(j_*(v\cap[F]))$. It has the property that $(f(H^{k+1}(F)),\psi+(-1)^{k+1}\psi^*)$ is the non-singular, $(-1)^{k+1}$-symmetric middle-dimensional cohomology intersection pairing of $F$.
\end{definition}

Fixing an orientation of an embedded normal bundle to $F$, denote $i^+,i^-\co F\to S^{n+2}\sm F$ small displacements in the positive and negative normal directions. A Meier--Vietoris argument shows that the reduced homology morphism \[i^+_*-i^-_*\co \widetilde{H}_r(F)\to \widetilde{H}_r(S^{n+2}\sm F)\] is an isomorphism for all $r$. Farber \cite[1.1]{MR718824}, then shows how to interpret the morphism \[e:=(\psi+(-1)^k\psi^*)^{-1}\psi\co  H^{k+1}(F)\to H^{k+1}(F)\] as the Poincar\'{e} dual of the morphism \[(i^+_*-i^-_*)^{-1}i^+_*\co H_{k+1}(F)\to H_{k+1}(F).\]

\begin{remark}The Seifert form is often defined in terms of these displacements $i^\pm$. The geometric definition of linking requires that the chains involved are disjoint, and the use of $i^\pm$ is one way to ensure this. Hence setting $a=j_*(u\cap[F])$ and $b=j_*(v\cap[F])$ in Definition \ref{def:seifert}, we have\[\psi([u],[v])=\text{link}(i^+_*a,b)=\text{link}(a,i^-_*b)=\text{link}(i^+_*a,i^-_*b).\]
\end{remark}

Levine \cite[pp.\ 43]{MR0461518} has made the connection between the Seifert and Blanchfield forms for a knot clear and we now quote these results for the reader's convenience. Suppose we are given a Seifert surface $F$ for $K$ and we remove an open normal neighbourhood of $F$ to perform a cut-and-paste construction of the infinite cyclic cover of $X_K$. There is a Meier--Vietoris sequence and it is shown by Levine \cite[p.\ 43]{MR0461518} that this breaks into short exact sequences in reduced homology\begin{equation}\label{eq:chnhtpy}0\to \widetilde{H}_r(F)[z,z^{-1}]\xrightarrow{(i^+_*z-i^-_*)}\widetilde{H}_r(S^{n+2}\sm F;\Z)[z,z^{-1}]\to \widetilde{H}_r(\overline{X_K})\to 0,\end{equation}the $\Z$-torsion free components of which give a short exact sequence\[0\to f(\widetilde{H}_r(F))[z,z^{-1}]\xrightarrow{(i^+_*-i^-_*)^{-1}(i^+_*z-i^-_*)}f(\widetilde{H}_r(F))[z,z^{-1}]\to f(\widetilde{H}_r(\overline{X_K}))\to 0,\]which, when $r=k+1$, is Poincar\'{e} dual to the sequence\[0\to f(H^{k+1}(F))[z,z^{-1}]\xrightarrow{-((1-e)+ez)}f(H^{k+1}(F))[z,z^{-1}]\to \underset{\cong f(H^{k+2}(\overline{X_K}))}{\underbrace{f(H^{k+2}(\overline{X_K},\overline{\partial X_K}))}}\to 0.\]So we can recover the underlying module of the Blanchfield form from the data of the Seifert form. What is more, the Blanchfield pairing on $f(H^{k+2}(\overline{X_K}))$ can also be completely recovered from the Seifert form. We refer the reader to Levine \cite[14.3]{MR0461518} for the geometric argument that this is the case.

\subsection*{Algebraic covering}

It is shown by Ranicki \cite{MR2058802} that method for recovering the Blanchfield form from a choice of Seifert form generalises to an entirely algebraic construction. This construction sends an $\eps$--symmetric Seifert form over $R$ to a non-singular $(-\eps)$--symmetric linking form over $(R[z,z^{-1}],P)$. We wish to use this construction to show that the double Witt group of Seifert forms is isomorphic to the double Witt group of Blanchfield forms, so we recall it now.

\begin{definition}[{Ranicki \cite{MR2058802}}]\label{def:coveringsief}
The \emph{covering} of an $R[s]$--module $K$ is the object of $\H(R[z,z^{-1}],P)$ given by \[B(K)=\coker((1-s)+sz\co K[z,z^{-1}]\to K[z,z^{-1}]),\]so that the action of $s$ on $B(K)$ is given by $(1-z)^{-1}$. The covering is a functor from f.g.\ $R[s]$--modules to $\H(R[z,z^{-1}],P)$, which respects the involutions.

The \emph{covering} of an $\eps$--symmetric Seifert form $(K,\psi)$ over $R$ is the non-singular $(-\eps)$--symmetric Blanchfield form $B(K,\psi)=(T,\lambda)$ over $R$ given by $T=B(K)$ (where the action of $s$ on $K$ is by $e=(\psi+\eps\psi^*)^{-1}\psi$), and\[\begin{array}{rcl}\lambda\co  T\times T&\to& P^{-1}R[z,z^{-1}]/R[z,z^{-1}];\\ ([x],[y])&\mapsto&-(1-z^{-1})(\psi+\eps\psi^*)(x,((1-e)+ez)^{-1}(y))\end{array}\]so that $\lambda$ is resolved by the following chain map\[\xymatrixcolsep{4pc}\xymatrix{
0\ar[r]&K[z,z^{-1}]\ar[rr]^{(1-e)+ez}\ar[d]^{(1-z)(\psi+\eps\psi^*)}&&K[z,z^{-1}]\ar[r]\ar[d]^-{-(1-z^{-1})(\psi+\eps\psi^*)}&T\ar[r]\ar[d]^-{\lambda}&0\\
0\ar[r]&K^*[z,z^{-1}]\ar[rr]^{((1-e)+ez)^*}&&K^*[z,z^{-1}]\ar[r]&T^\wedge\ar[r]&0}\](The $(1-z)$ factor forces $(-\eps)$--symmetry.)
\end{definition}

In \cite[1.8(i)]{MR2058802} and \cite[3.10]{MR2058802}, Ranicki shows that covering defines a surjective morphism from the monoid of isomorphism classes of $\eps$--symmetric Seifert forms over $R$ to monoid of isomorphism classes of $(-\eps)$--symmetric linking forms over $(R[z,z^{-1}],P)$.

\begin{lemma}\label{lagrangians}The covering morphism sends a lagrangian $j\co L\hookrightarrow K$ of a non-singular Seifert form $(K,\psi)$ to a lagrangian of the covering linking form. Hence the covering morphism preserves the property of being metabolic and of being hyperbolic.
\end{lemma}

\begin{proof}We wish to show that the sequence\[0\to B(L,e)\xrightarrow{B(j)} T\xrightarrow{B(j)^\wedge\circ\lambda} B(L,e)^\wedge\to 0\] is exact. This sequence is resolved by the commutative diagram\[\xymatrix{
0\ar[r]&L[z,z^{-1}]\ar[rrr]^{(1-e)+ez}\ar[d]^{j}&&&L[z,z^{-1}]\ar[r]\ar[d]^-{j}&B(L,e)\ar[r]\ar[d]^{B(j)}&0\\
0\ar[r]&K[z,z^{-1}]\ar[rrr]^{(1-e)+ez}\ar[d]^{\psi+\eps\psi^*}&&&K[z,z^{-1}]\ar[r]\ar[d]^-{\psi+\eps\psi^*}&B(K,e)\ar[r]\ar[d]_\cong^{B(\psi+\eps\psi^*)}&0\\
0\ar[r]&K^*[z,z^{-1}]\ar[d]^{(1-z)}\ar[rrr]^{e^*+(1-e^*)z}&&&K^*[z,z^{-1}]\ar[d]^{-(1-z^{-1})}\ar[r]&B(K^*,1-e^*)\ar[r]\ar[d]_\cong^\xi&0\\
0\ar[r]&K^*[z,z^{-1}]\ar[d]^{j^*}\ar[rrr]^{((1-e)+ez)^*}&&&K^*[z,z^{-1}]\ar[d]^{j^*}\ar[r]&B(K,e)^\wedge\ar[r]\ar[d]^{B(j)^\wedge=B(j^*)}&0\\
0\ar[r]&L^*[z,z^{-1}]\ar[rrr]^{((1-e)+ez)^*}&&&L^*[z,z^{-1}]\ar[r]&B(L,e)^\wedge\ar[r]&0}\]where $\xi$ is defined to make the diagram commute. As the sequence of $R$--modules \[0\to L\xrightarrow{j} K\xrightarrow{j^*\circ(\psi+\eps\psi^*)} L^*\to 0\]is exact, we need only show that $\ker(B(j^*))=\ker(B(j^*)\circ\xi)$ so that inserting the isomorphism $\xi$ preserves exactness of the sequence\[0\to B(L,e)\xrightarrow{B(j)} B(K,e) \xrightarrow{B(j^*)\circ B(\psi+\eps\psi^*)} B(L,e)\to 0.\]But $j^*$ commutes with the maps $(1-z)$ and $-(1-z^{-1})$, so the result follows.
\end{proof}

\begin{theorem}\label{thm:covering}The covering induces an isomorphism of double Witt groups\[B\co \widehat{DW}_\eps(R)\xrightarrow{\cong} DW^{-\eps}(R[z,z^{-1}],P).\]
\end{theorem}

\begin{proof}Lemma \ref{lagrangians} shows that covering morphism is a well-defined homomorphism on the level of double Witt groups. We have already noted that it is surjective. So it will be sufficient to prove that if the covering of a Seifert form vanishes then the Seifert form is hyperbolic.

In general, if the covering $B(K)$ of a f.g.\ projective $R$--module $K$ with an $R[s]$--module structure given by $e$ vanishes, then $e(1-e)\co K\to K$ is nilpotent (see Ranicki \cite[2.3]{MR2058802}). This is called being a \emph{near-projection}. This is equivalent to the condition that there is a decomposition of f.g.\ projective $R$--modules with $R[s]$--module structure:\begin{equation}\label{eq:direct}(K,e)\cong (K_+,e_+)\oplus (K_-,e_-),\end{equation} with both $e_-$ and $1-e_+$ nilpotent. It is shown moreover in \cite[1.8(i)]{MR2058802} that there exists $k>0$ such that, defining a projection \[p_e:=(e^k+(1-e)^k)^{-1}e^k\co K\to K,\] the direct sum decomposition of equation \ref{eq:direct} may be taken to be $K_+=\im(p_e)$, $K_-=\im(1-p_e)$. We will improve this to a decomposition of Seifert forms.

Note that the $R[s]$--module $(K,e)$ being a near-projection is equivalent to any one of the $R[s]$--modules $(K,1-e)$, $(K^*,e^*)$, $(K^*,1-e^*)$ being a near projection. Hence we may similarly decompose the Seifert dual module\[(K,e)^*=(K^*,1-e^*)=(K^*,1-e^*)_+\oplus (K^*,1-e^*)_-\]using the projection $p_{1-e^*}\co K^*\to K^*$. But as \[p_e^*= ((e^*)^k+(1-e^*)^k)^{-1}(e^*)^k=p_{e^*}=1-p_{1-e^*}\] and $(\im(p_e))^*=\im(p_e^*)$ we can identify the summands\[\begin{array}{rcl}(K^*,1-e^*)_+&=&((K_-)^*,(1-e^*)|_{(K_-)^*})\\ (K^*,1-e^*)_-&=&((K_+)^*,(1-e^*)|_{(K_+)^*})\end{array}\]Now if $(K,\psi)$ is a Seifert form such that $B(K,e)=0$, then as $\psi e=(1-e^*)\psi$ we have an isomorphism of Seifert forms \[\lmat p\\ 1-p\rmat\co \left((K,e),\psi\right)\xrightarrow{\cong}\left((K_+\oplus K_-,e_+\oplus e_-),\lmat 0&\psi|_{K_-}\\ \psi|_{K_+}&0\rmat\right)\] so that $(K,\psi)$ is hyperbolic.
\end{proof}

\begin{example}\label{ex:Cseif}If $R$ is a field with involution then $P$ is coprime to the multiplicative subset $S=(1-z)^\infty$ and such that $SP=R[z,z^{-1}]\sm\{0\}$, the full multiplicative subset. By Theorem \ref{thm:multisignature} we have that\[DW^\eps(R[z,z^{-1}],R[z,z^{-1}]\sm\{0\})\cong DW^\eps(R[z,z^{-1}],P)\oplus DW^\eps(R[z,z^{-1}],(1-z)^\infty).\]Hence, for instance, using Examples \ref{ex:C} and \ref{ex:R} we have \[DW^\eps(\C[z,z^{-1}],P)\cong\bigoplus_{a\in S^1\sm\{1\}}\bigoplus_{l=1}^\infty\Z\]and\[DW^\eps(\R[z,z^{-1}],P)\cong\left\{\begin{array}{lll}\bigoplus_{0<\theta<\pi}\bigoplus_{l=1}^\infty\Z&&\eps=1,\\
(\bigoplus_{l=1}^\infty\Z)\oplus(\bigoplus_{0<\theta<\pi}\bigoplus_{l=1}^\infty\Z)&&\eps=-1.\end{array}\right.\]
\end{example}

\bibliographystyle{plain}
\def\MR#1{}
\bibliography{writeup}

\end{document}